\documentclass{article}
\usepackage{amssymb,amsmath,amsthm,graphicx,cite}

\def\qed{\hfill {\hbox{${\vcenter{\vbox{               
   \hrule height 0.4pt\hbox{\vrule width 0.4pt height 6pt
   \kern5pt\vrule width 0.4pt}\hrule height 0.4pt}}}$}}}

\newtheorem{theorem}{Theorem}

\newtheorem{proposition}[theorem]{Proposition}
\newtheorem{corollary}[theorem]{Corollary}

\theoremstyle{definition}
\newtheorem{example}{Example}
\newtheorem{definition}{Definition}
\newtheorem{remark}{Remark}

\date{}

\title{\Large \textbf{Bikei Invariants and Gauss Diagrams for Virtual Knotted 
Surfaces}}

\author{Sam Nelson\footnote{Email: knots@esotericka.org. Partially supported by Simons Foundation collaboration grant 316709}\and
Patricia Rivera\footnote{Email: patriciariv25@gmail.com}}

\begin{document}
\maketitle

\begin{abstract}
Marked vertex diagrams provide a combinatorial way to represent knotted 
surfaces in $\mathbb{R}^4$; including virtual crossings allows for a 
theory of virtual knotted surfaces and virtual cobordisms. Biquandle counting
invariants are defined only for marked vertex diagrams representing 
knotted orientable surfaces; we extend these invariants to all virtual
marked vertex diagrams by considering colorings by involutory biquandles, 
also known as bikei. We introduce a way of representing marked vertex diagrams
with Gauss diagrams and use these to characterize orientability.
\end{abstract}

\parbox{5.5in} {\textsc{Keywords:} marked vertex diagrams, ch-diagrams, 
knotted surfaces, virtual knotted surfaces, bikei

\smallskip

}

\section{\large\textbf{Introduction}}\label{I}

\textit{Biquandles} are algebraic structures with axioms motivated by the 
oriented Reidemeister moves. In particular, the set of labelings of the 
semiarcs in a oriented knot diagram by elements of a finite biquandle 
satisfying a certain labeling condition forms a computable 
invariant known as the \textit{biquandle counting invariant}. Biquandles 
were introduced in \cite{FRS} and have been much studied in recent years; 
see \cite{FJK,KR} etc. for more.

\textit{Marked vertex diagrams}, also known as
\textit{ch-diagrams}, are planar diagrams similar to ordinary knot 
diagrams which encode a knotted surface in $\mathbb{R}^4$ \cite{LO,Y,L}. 
It has been recently established (\cite{K}) that two marked vertex diagrams 
represent ambient isotopic surfaces in $\mathbb{R}^4$ if and only if they 
are related by a sequence of the \textit{Yoshikawa moves} introduced in 
\cite{Y}. Marked vertex diagrams can be used to 
represent both closed knotted surfaces and cobordisms between knots and
have advantages over some other popular methods of representing knotted
surfaces in $\mathbb{R}^4$: they are easier to draw than broken surface 
diagrams and require only a single diagram, unlike movie diagrams which
require multiple diagrams. See \cite{LO,JKL,L,Y} for more.

\textit{Virtual knots}, also known as \textit{abstract knots}, are a 
combinatorial generalization of knots including \textit{virtual crossings}
representing genus in the ambient space of the knot. Including virtual 
crossings in marked vertex diagrams yields \textit{virtual knotted surfaces}; 
see \cite{K2,K3,K4} for more.

\textit{Bikei}, also known as \textit{involutory biquandles}, were introduced
in \cite{AN} as a special case of biquandles which can be used to define
invariants of unoriented knots and links. 
Unlike the more general case of biquandles, bikei counting invariants are 
defined for all virtual marked vertex diagrams, regardless of topological type.
In this paper, we study bikei counting invariants for virtual knotted surfaces.

\textit{Gauss diagrams} (and equivalently \textit{Gauss codes}) are a 
computer-friendly way of representing virtual knot diagrams. We define Gauss 
diagrams for marked vertex diagrams and use these to characterize orientability.

The paper is organized as follows. In Section \ref{B} we review the basics 
of bikei. In Section \ref{KS} we review
knotted surfaces and marked diagrams. In Section \ref{RBI} we 
define bikei counting invariants for marked vertex diagrams. In Section 
\ref{G} we introduce Gauss diagrams for marked vertex diagrams and 
use them to characterize orientability. We end in Section \ref{Q} with
some questions for future research.

\section{\large\textbf{Bikei}}\label{B}

We begin with a standard definition. 

\begin{definition}
Let $X$ be a set. A \textit{biquandle structure} on $X$ is a pair of binary 
operations $(x,y)\mapsto x^y, x_y$ such that for all $x,y,z\in X$ we have
\begin{itemize}
\item[(i)] $x^x=x_x$,
\item[(ii)] the maps $\alpha_y,\beta_y:X\to X$ and $S:X\times X\to X\times X$
defined by $\alpha_y(x)=x_y$, $\beta_y(x)=x^y$ and 
$S(x,y)=(y_x,x^y)$
are bijective, and
\item[(iii)] the \textit{exchange laws}
\[\begin{array}{rcl}
(x^y)^{(z^y)} & = & (x^z)^{(y_z)} \\
(x^y)_{(z^y)} & = & (x_z)^{(y_z)} \\
(x_y)_{(z_y)} & = & (x_z)_{(y^z)}
\end{array}.\]
\end{itemize}
A biquandle such that $x_y=x$ for all $x,y$ is a \textit{quandle}. A biquandle
such that $(x^y)^y=x=(x_y)_y$, $x^{yy_x}=x^y$ and $y_{x^y}=y_x$ for all $x,y\in X$
is a \textit{bikei}.
\end{definition}

\begin{remark}
We are using the original notation for biquandles from \cite{FRS}; much
recent work uses similar-looking but different notation resulting in 
less symmetric axioms.
\end{remark}

We will be primarily interested in the case of bikei.

\begin{example} Let $n\in \mathbb{Z}.$
Any group $G$ is a bikei with $x^y=yx^{-1}y$ and $x_y=x$ as well as with
$x_y=x$ and $x^y=yx^{-1}y$.
\end{example}

\begin{example}
Any abelian group $A$ with elements $s,t$ satisfying $s^2=t^2=1$
and $1+t=s(1+t)$ is a bikei with  
\[x^y=tx+(s-t)y,\quad x_y=s x\]
Such a bikei is called an \textit{Alexander bikei}.
\end{example}

Bikei structures on a set $X=\{x_1,\dots, x_n\}$ can be specified using
an $n\times 2n$ matrix encoding the operation tables of the biquandle
operations. More precisely, let $M_{i,j}=k$ where
\[x_k =\left\{\begin{array}{ll}
(x_i)^{x_j} &\ 1\le j\le n \\
(x_i)_{x_j} &\ n+1\le j\le 2n \\
\end{array}\right.\]

\begin{example}
The Alexander bikei $X=\mathbb{Z}_4=\{x_1=1,x_2=2,x_3=3,x_4=0\}$ 
with $s=3$ and $t=1$ has biquandle matrix
\[\left[\begin{array}{rrrr|rrrr}
3 & 1 & 3 & 1 & 3 & 3 & 3 & 3 \\
4 & 2 & 4 & 2 & 2 & 2 & 2 & 2 \\
1 & 3 & 1 & 3 & 1 & 1 & 1 & 1 \\
2 & 4 & 2 & 4 & 4 & 4 & 4 & 4
\end{array}\right]\]
\end{example}

The bikei axioms come from the unoriented Reidemeister moves. We think of
elements of $X$ as labels for the semiarcs in an unoriented  knot or link
diagram with operations as depicted.
\[\includegraphics{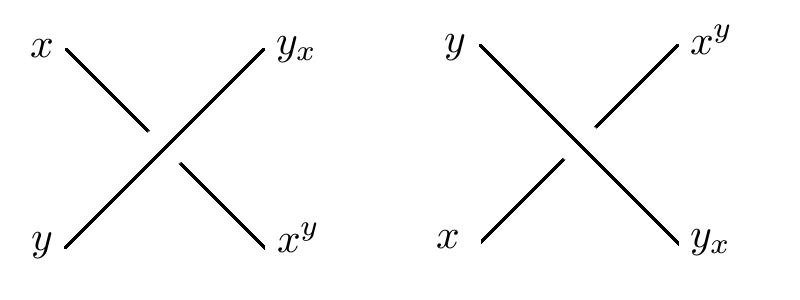}\]
Then the Reidemeister I move requires that for all $x\in X$, we have
$x^x=x_x$:
\[\includegraphics{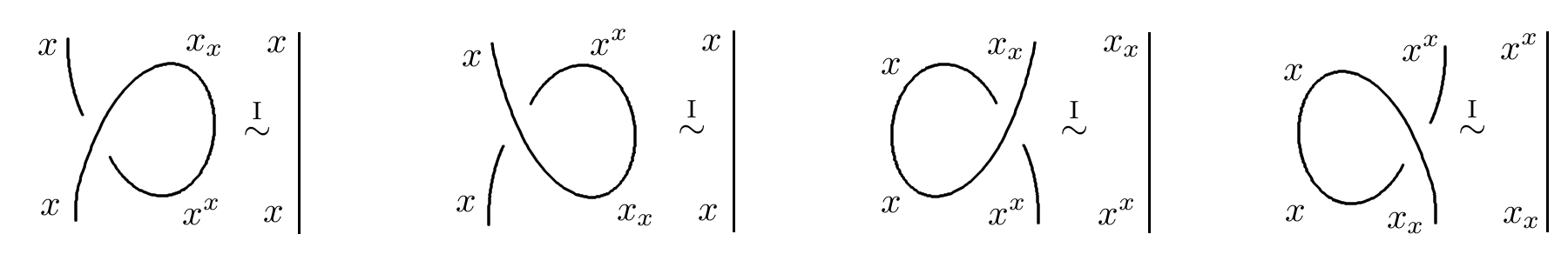}\]
The conditions $(x^y)^y=(x_y)_y=x$,  $x^y=x^{y_x}$  and $x_y=x_{y^x}$
satisfy the Reidemeister II moves.
\[\scalebox{0.95}{\includegraphics{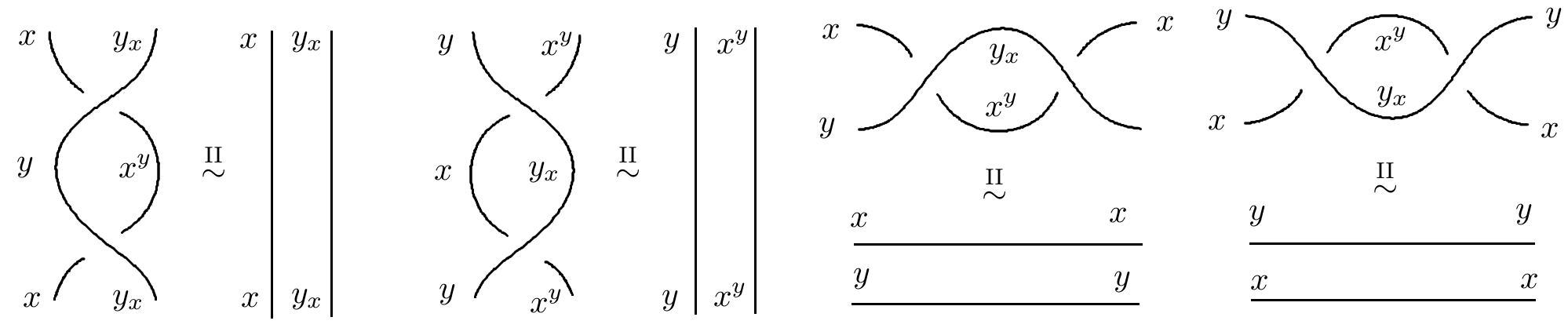}}\]
The Reidemeister III move implies the exchange laws:
\[\includegraphics{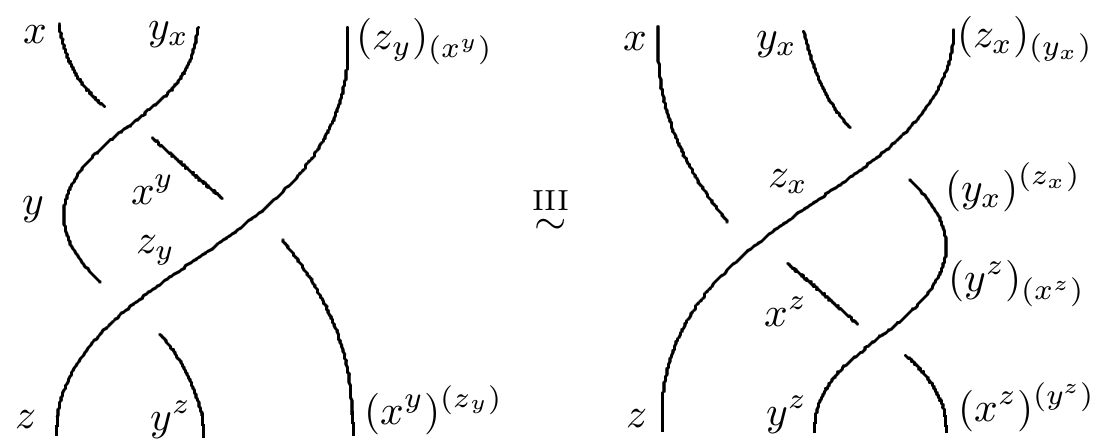}\]

For oriented knots and links, biquandle colorings by a finite biquandle define
a computable invariant known as the \textit{biquandle counting invariant}.
For unoriented knots and links $L$, the number of bikei colorings of a diagram
of $L$ by a finite bikei $X$  is an invariant known as the \textit{bikei 
counting invariant}, denoted $\Phi_X^{\mathbb{Z}}(L).$ We will extend this 
invariant to virtual knotted surfaces in following sections.

\section{\large\textbf{Knotted Surfaces}}\label{KS}

A \textit{knotted surface} is a smoothly embedded compact surface 
$\Sigma\subset \mathbb{R}^4$ with finitely many components. Knotted surfaces 
can be represented with \textit{broken surface diagrams} analogous to
knot diagrams where a broken sheet indicates the sheet crosses under
in the fourth dimension much as a broken strand in a knot diagram indicates
the strand crossing under in the third dimension. The self-intersection set
in the projection into $\mathbb{R}^3$ can include endpoints and triple points
as well as closed curves. Alternatively, knotted surfaces
can be represented with \textit{movie diagrams} and \textit{braid charts};
see \cite{CKS} for much more.
\[\includegraphics{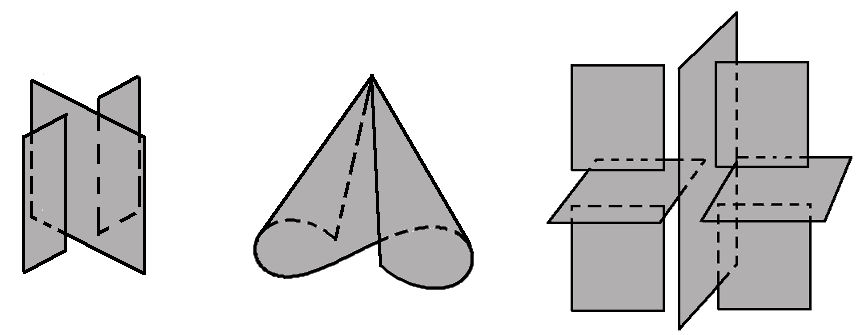}\]

Every knotted surface can be moved by ambient isotopy into a position such
that all of its maxima in the $x_3$ (say) direction are in the $x_3=1$ 
hyperplane, all of its minima are in the $x_3=-1$ hyperplane and all of its 
saddle points are in the $x_3=0$ hyperplane; such a position is known as 
a \textit{hyperbolic splitting}. In particular each cross-section of the 
surface by a hyperplane of the form $x_3=\epsilon$ for 
$\epsilon\in(-1,0)\cup(0,1)$ is an unlink, with the $(x_1,x_2)$ 
cross-sections forming Reidemeister move sequences ending with crossingless
unlink diagrams near $\epsilon=\pm 1$. 

We can represent a knotted surface in such a position with a \textit{marked 
vertex diagram} or
\textit{ch-diagram}, a diagram consisting of ordinary crossings together with
\textit{saddle crossings} representing saddle points:
\[\includegraphics{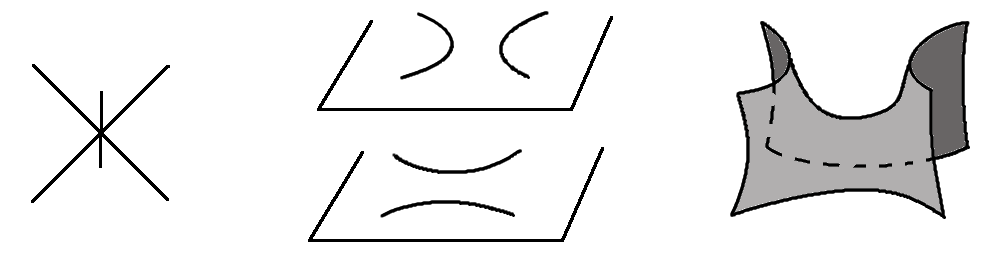}\]
To recover a broken surface diagram from a marked vertex
diagram, we resolve the 
saddle crossings into saddles with the crossings resolving into crossed sheets;
near the $x_3=0$ hyperplane, we have a cobordism between unlinks. These unlinks
then resolve over the intervals $x_3\in(-1,0)\cup (0,1)$ into disjoint circles
which we cap off with maxima and minima. See \cite{L,LO,Y} for more.
\[\includegraphics{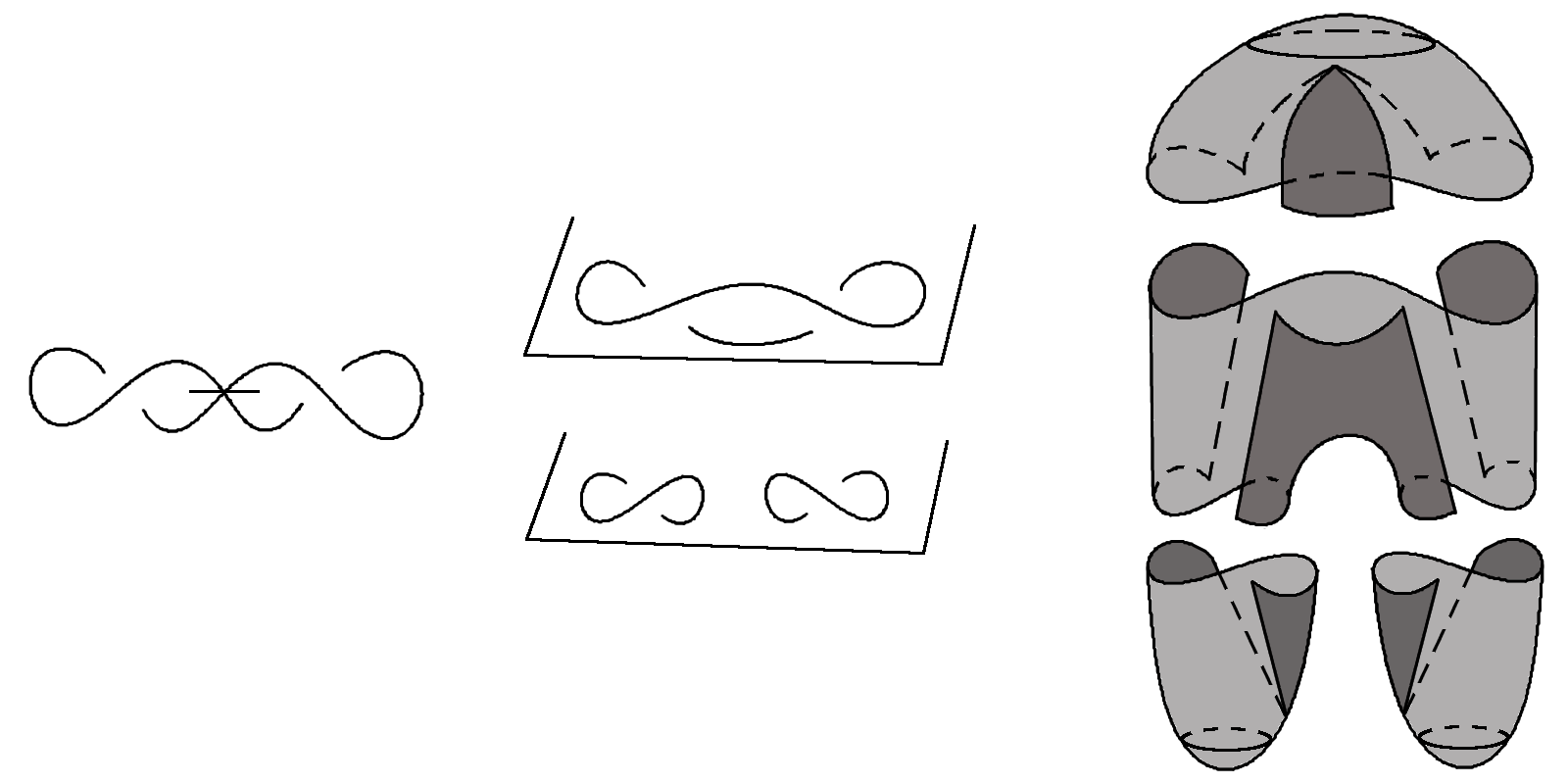}\]

\begin{remark}
Marked vertex diagrams which yield nontrivial knots or links $L,L'$ after 
smoothing the saddle crossings do not 
correspond to knotted closed surfaces but to \textit{corbordisms} between 
$L$ and $L'$, i.e., surfaces whose boundaries are $L\cup L'$.
\end{remark}

In \cite{Y}, Yoshikawa introduced a set of moves now known as \textit{Yoshikawa 
moves} and conjectured that two knotted surfaces are ambient isotopic iff their
marked vertex diagrams are related by the Yoshikawa moves. This conjecture has recently
been established \cite{K}. There are eight moves, the first three of which
are the usual Reidemeister moves:
\[\includegraphics{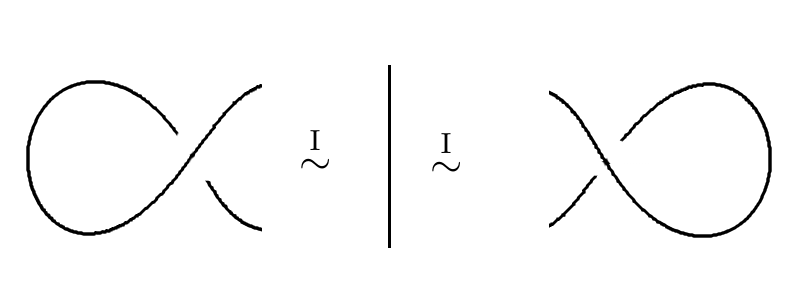}\quad\quad\quad
\includegraphics{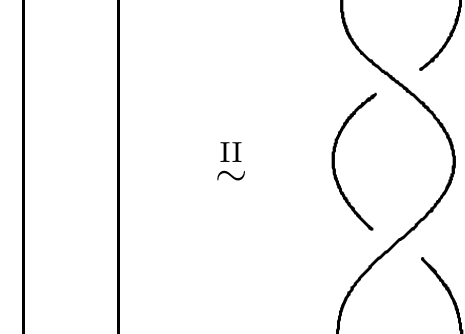}\]
\[\includegraphics{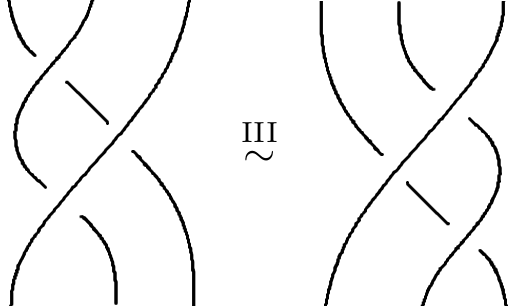}\quad\quad\quad
\includegraphics{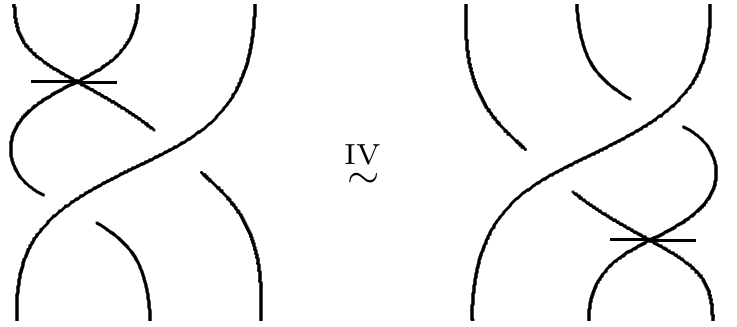}\]
\[\includegraphics{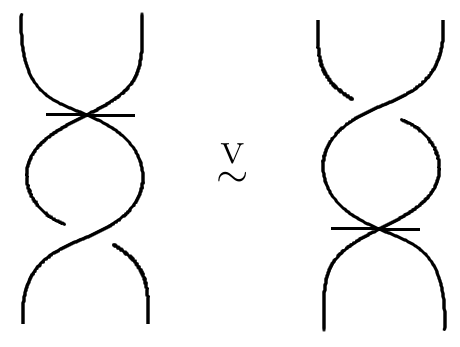}\quad\quad\quad
\includegraphics{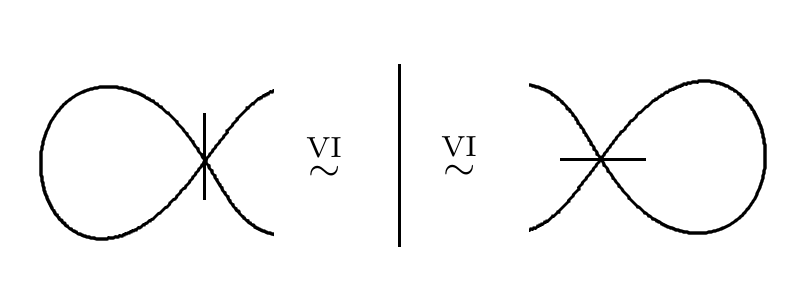}\]
\[\includegraphics{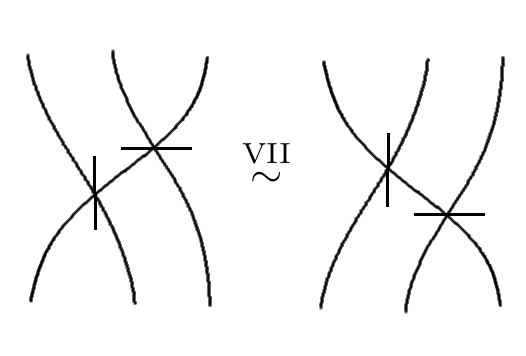}\quad\quad\quad
\includegraphics{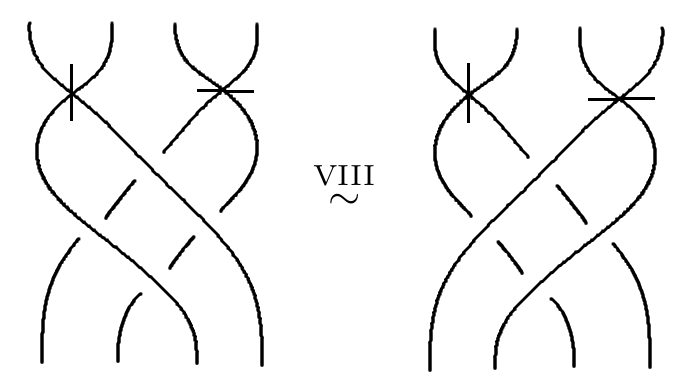}\]

In \cite{Y}, marked vertex diagrams are required to yield unlink diagrams 
after smoothing
the saddles in order to obtain closed knotted surfaces; as we have observed, 
relaxing this requirement yields cobordisms between the links obtained after 
smoothing. In \cite{K4} this idea is used to extend the notion of cobordism 
to the case of \textit{virtual knots}, knots and links with \textit{virtual 
crossings}, which represent genus in the ambient space rather than points where 
the knot is close to itself within the ambient space. We will not break our
semiarcs at virtuals crossings, so the biquandle labeling rule is as depicted.
\[\includegraphics{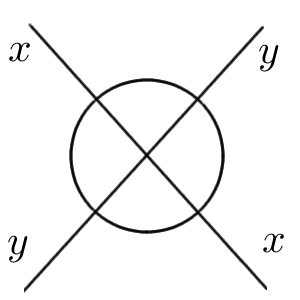}\]
Virtual crossings 
interact with classical and saddle crossings via the \textit{detour move}
\[\includegraphics{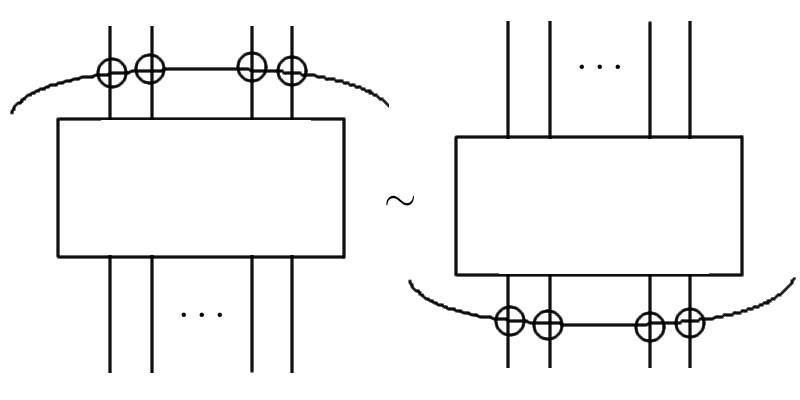}\]
which says that any portion of the knot with only virtual crossings can be 
replaced with any other strand with the same endpoints and only virtual 
crossings.
In particular, marked vertex diagrams with virtual crossings represent 
cobordisms between the virtual links obtained by smoothing the saddle crossings.

Finally, we note that there is an analog of crossing number for virtual marked 
vertex diagrams; in \cite{Y}, the minimal number of classical crossings for a
given surface-knot type $K$ is denoted $c(K)$, and the minimal number of saddle
crossings is denoted $h(K)$ (for ``hyperbolic points''); then the 
\textit{$ch$-index} is $ch(K)=c(K)+h(K)$. 

\begin{definition}
The \textit{virtual ch-index} of a virtual knotted surface $K$, $vch(K)$, is 
the minimal total number of virtual, classical, and saddle crossings over all 
diagrams $D$ of $K$. 
\end{definition}


\section{\large\textbf{Bikei Counting Invariants}}\label{RBI}

Let $X$ be a bikei and $L$ a marked vertex diagram. An $X$-labeling 
of $L$ is an assignment of elements of $X$ to the semiarcs of $L$ (the portions
of the diagram between the crossing points) satisfying the rules pictured below:
\[\includegraphics{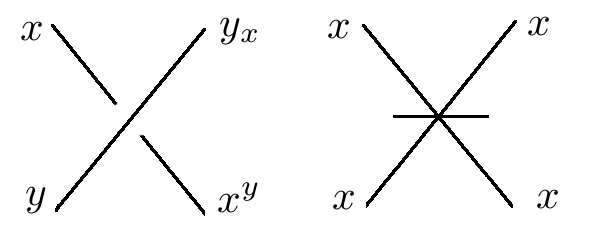}\]

Checking the remaining Yoshikawa moves, we have
\[\includegraphics{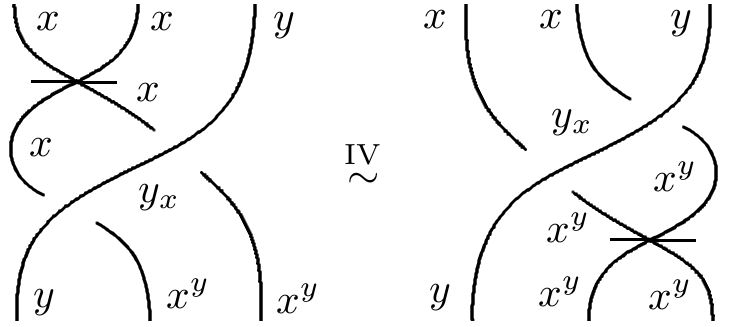}\quad \includegraphics{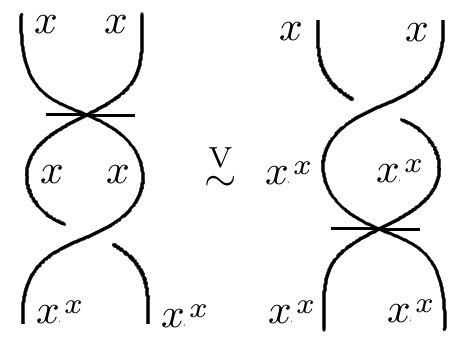}\]
\[\includegraphics{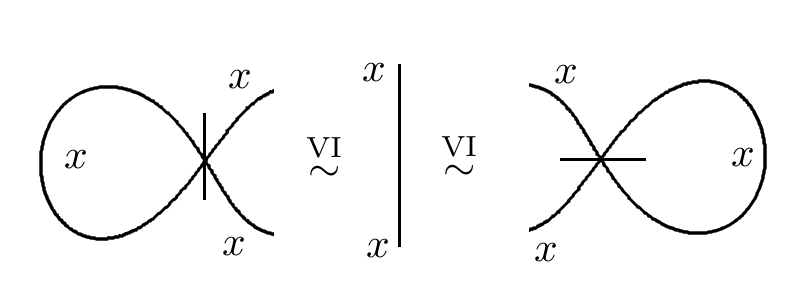}\quad \includegraphics{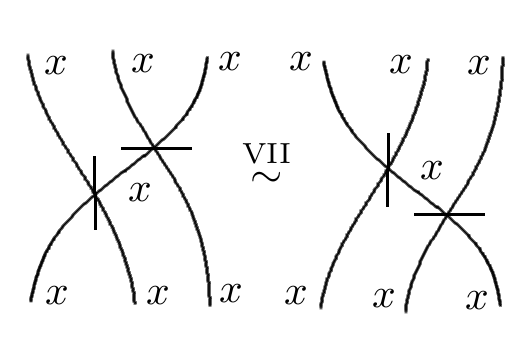}\]
\[\includegraphics{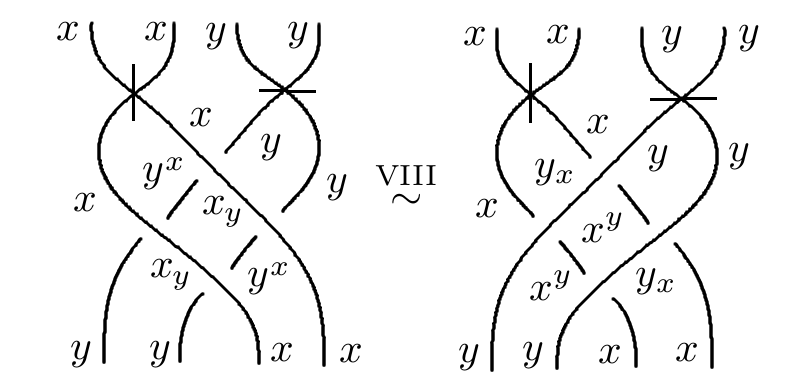}\]
Moreover, $X$-labelings are unaffected by the detour move.Thus we obtain:

\begin{theorem}
If $X$ is a bikei and $L$ and $L'$ are marked vertex diagrams related by 
Yoshikawa moves then for every 
$X$-labeling of $L$ there is a unique corresponding $X$-labeling of $L'$.
\end{theorem}

\begin{corollary}
Let $X$ be a finite bikei. The number of $X$-labelings of a virtual
marked vertex diagram is an invariant of virtual knotted surfaces.
\end{corollary}

\begin{example}
The diagram $8_1$ 
below represents the \textit{spun trefoil}, the surface swept out
in $\mathbb{R}^4$ by spinning a trefoil knot about an axis. 
\[\includegraphics{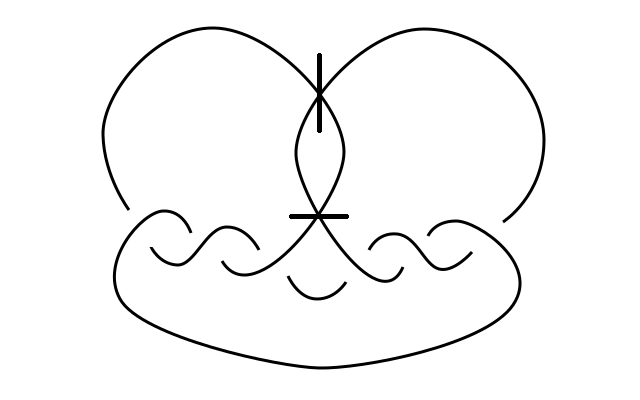}\]
It is distinguished from the unknotted sphere $0_1$
by the bikei counting invariant with bikei $X$ given by the matrix
\[\left[\begin{array}{rrrr|rrrr}
1 & 1 & 1 & 1 & 1 & 1 & 1 & 1 \\
3 & 2 & 2 & 2 & 3 & 2 & 4 & 3 \\
2 & 3 & 3 & 3 & 2 & 4 & 3 & 2 \\
4 & 4 & 4 & 4 & 4 & 3 & 2 & 4
\end{array}\right]\]
with $\Phi_X^{\mathbb{Z}}(8_1)=10\ne 4= \Phi_x^{\mathbb{Z}}(0_1).$
\end{example}

One important difference between surface knot theory and classical knot theory
is that in classical knot theory, there is only one object to knot, namely the 
circle $S^1$ (or disjoint unions of copies of $S^3$), while in surface knot 
theory there are already infinitely many topologically distinct surfaces 
before knotting, characterized by cross-cap number $c$ and genus $g$. 
\[\includegraphics{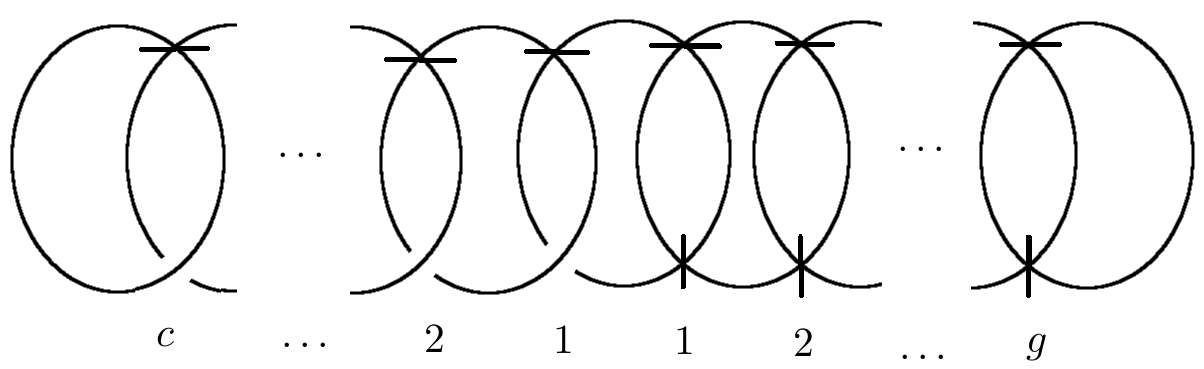}\]

An orientation for a virtual marked vertex diagram is a choice of orientation 
for each semiarc such that at classical crossing we have a ``pass-through''
rule at at each saddle crossing we have a ``source-sink rule'':
\[\includegraphics{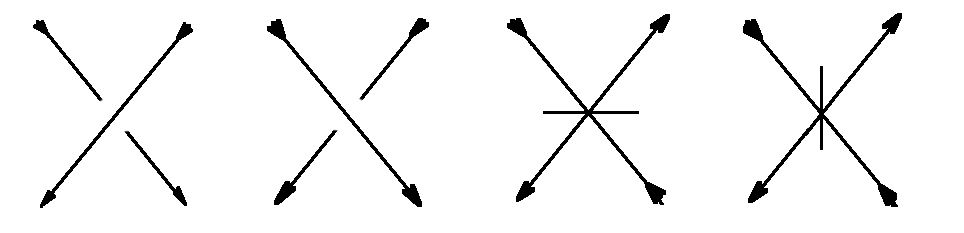}\]
These define a local orientation on the surface by defining a choice of normal
vector on each sheet.
\[\includegraphics{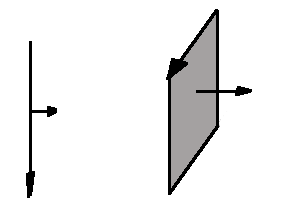}\]
If the surface is not orientable, these local orientations cannot be 
extended to a global orientation.

\begin{definition}
The closed curves obtained by ``passing through'' both saddle and classical 
crossings in a virtual marked vertex diagram are its \textit{na\"ive 
components}.
\end{definition}

\begin{remark}\label{rem1}
We note that the number of na\"ive components is not invariant under Yoshikawa
moves, since the following move splits or merges two na\"ive components:
\[\includegraphics{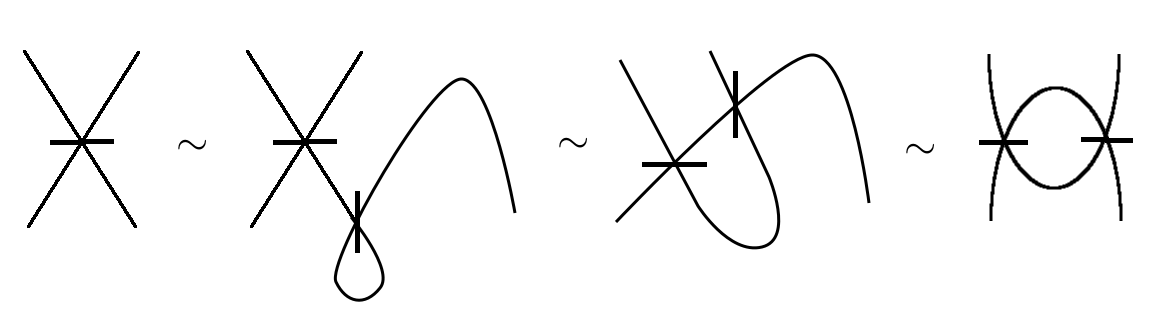}\]
See also \cite{Y}.
\end{remark}

Since an unknotted orientable surface $S$ of genus $g$ has a diagram with
only saddle crossings, every coloring of $S$ by any bikei $X$ is monochromatic. 
Thus we have:

\begin{proposition}
For any unknotted orientable surface $S$ of genus $g\in\mathbb{N}$ and bikei
$X$, the bikei counting invariant of $S$ is $\Phi_X^{\mathbb{Z}}=|X|$.
\end{proposition}

For non-orientable unknotted surfaces, we likewise need monochromatic colorings,
but the classical crossings present an obstruction: a monochromatic coloring
of such a surface by $x\in X$ is only valid if $x=x^x=x_x$. Then we have:

\begin{proposition}
For any unknotted non-orientable surface $S$ of genus $g\in\mathbb{N}$ 
and cross-cap number $c$ and bikei
$X$, the bikei counting invariant of $S$ is $\Phi_X^{\mathbb{Z}}=|F|$ where
\[F=\{x\in X\ |\ x^x=x_x=x\}.\]
\end{proposition}

\begin{corollary}
If a bikei $X$ is a \textit{kei}, i.e. an involutory quandle, then
$\Phi_X^{\mathbb{Z}}=|X|$ for all unknotted surfaces $S$.
\end{corollary}

The set $F=\{x\in X\ |\ x^x=x_x=x\}$ is the union of singleton sub-bikeis
of $X$. $F$ can be all of $X$ (as in the quandle case where $x^x=x$ for all 
$x\in X$), or it could be a proper sub-bikei; it can even be empty.


\begin{example}\label{ex2col}
Let $X=\mathbb{Z}_2$ and define $x^y=x_y = x+1$. Then $\Phi_X^{\mathbb{Z}}(S)=2$
for all unknotted orientable surfaces and $\Phi_X^{\mathbb{Z}}(S)=0$ for all
unknotted non-orientable surfaces. A coloring by $X$ will be called a 
\textit{$2$-coloring}.
\end{example}

Finally, we note that classical crossing changes do not affect 2-colorability.
Thus, we have:

\begin{proposition}\label{prop:2c}
Knotted orientable surfaces which can be unknotted by classical crossing 
changes are 2-colorable; knotted non-orientable surfaces which can be 
unknotted by classical crossing changes are not 2-colorable.
\end{proposition}

\section{Gauss diagrams for virtual knotted surfaces}\label{G}

In light of proposition \ref{prop:2c} it is natural to ask whether a virtual
knotted surface is orientable if and only if it is 2-colorable. For virtual
knots, it is known that classical crossing change is not an unknotting
operation; for instance, the \textit{Kishino virtual knot} below cannot be
unknotted by classical crossing change.
\[\includegraphics{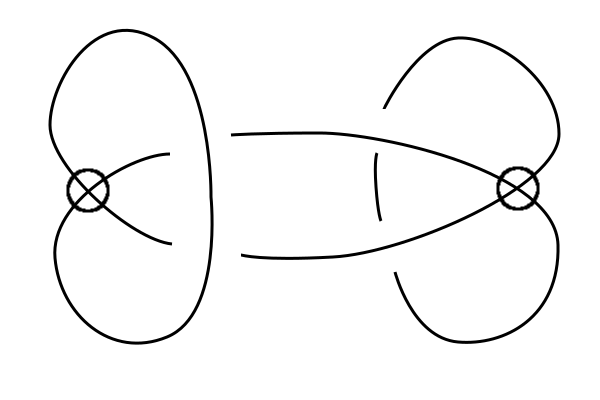}\]

However, in the virtual knotted case orientability and 2-colorability need not 
coincide: the diagram below cannot be given a consistent source-saddle 
orientation, but is 2-colorable:
\[\includegraphics{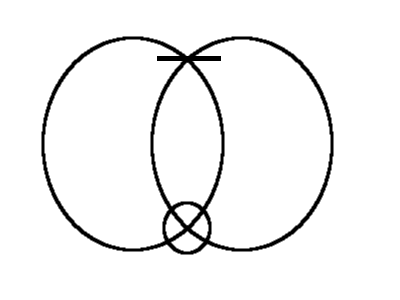}\]

To characterize orientability for virtual knotted surface diagrams, we
introduce Gauss diagrams for virtual marked vertex diagrams.
In light of the move in remark \ref{rem1}, we can without loss of generality 
consider only diagrams with a single na\"ive component, which we write as a 
circle
around which we write crossing labels representing over/under crossings points
and two points representing each saddle point. The two instances of each 
crossing are then connected by chords, with the classical crossing chords 
directed toward the undercrossing point and the saddle crossing chords
marked to indicate the direction of the saddle as depicted.
\[\includegraphics{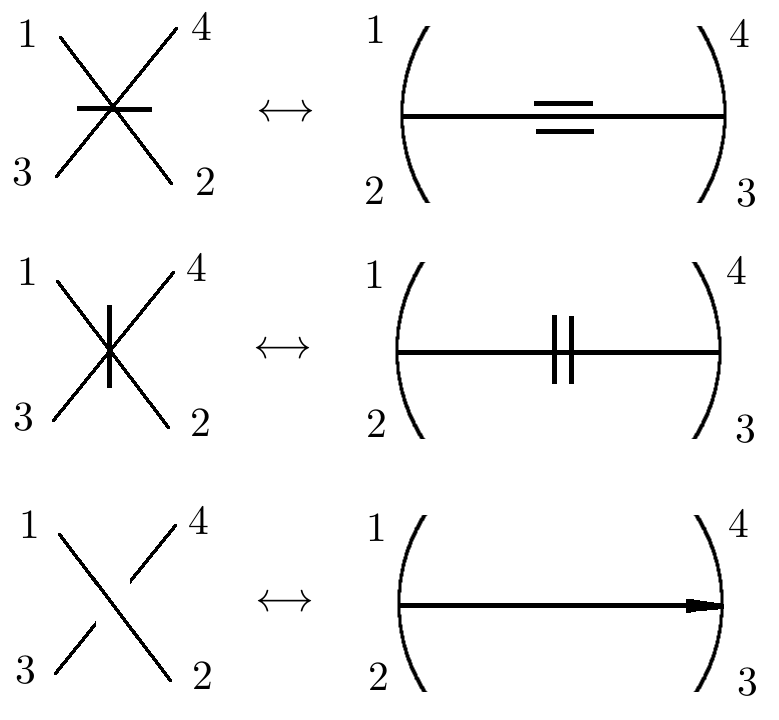}\]
The portions of the circles between endpoints of either chords or arrows 
correspond to the semiarcs of the diagram.
\[\includegraphics{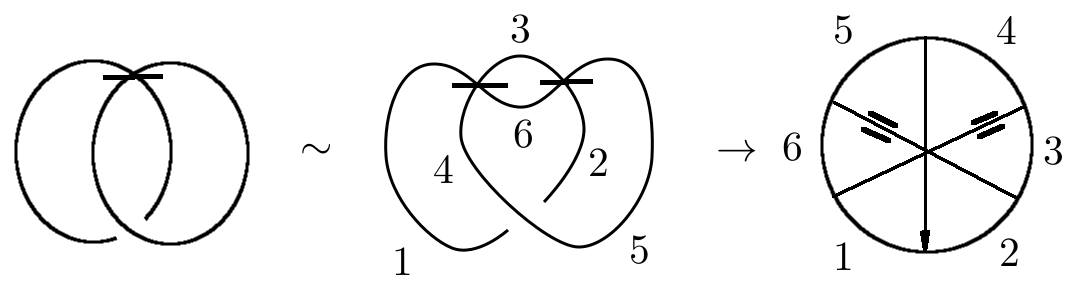}\]

We then have the following characterization of orientability:

\begin{proposition}
A marked vertex Gauss diagram is orientable iff for every 
chord $C$ the number of chord endpoints between the two endpoints of $C$
is even.
\end{proposition}

\begin{proof}
Traveling around the naive component of an oriented diagram, each chord 
endpoint reverses the local orientation. Since each chord represents a
source-sink oriented saddle crossing, the local picture must be as depicted:
\[\includegraphics{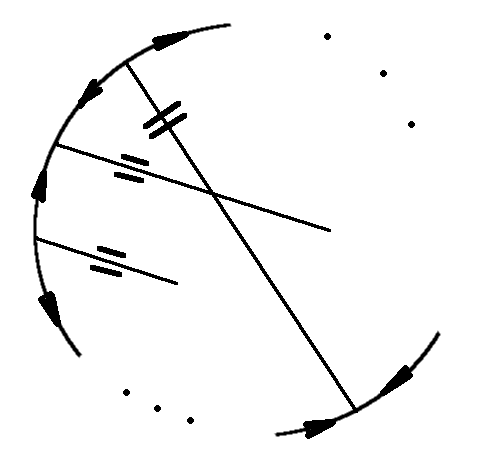}\]
Thus, we need an even
number of switches between the ends of $C$ on both sides.

Conversely, if every cord meets the stated criteria, then there are two 
choices of global orientation for the diagram, determined by choosing a 
direction for one semiarc and alternating at every chord endpoint; the 
condition that every chord has an even number of endpoints between its two ends 
implies that both such orientations are source-sink for each saddle crossing
and pass-through for each classical crossing.
\end{proof}

\section{\large\textbf{Questions}}\label{Q}

We conclude with some questions for future research.
\begin{itemize}
\item What kinds of enhancements of the bikei counting invariant are there? 
Cocycle enhancements, for instance, generally require orientation.
\item What happens when we add nontrivial operations at the virtual crossings?
\end{itemize}

\bibliography{sn-pr}{}
\bibliographystyle{abbrv}

%

%
%
%
%
%
%
%
%
%
%

\bigskip

\noindent
\textsc{Department of Mathematical Sciences \\
Claremont McKenna College \\
850 Columbia Ave. \\
Claremont, CA 91711}

\end{document}